\documentclass[a4paper,12pt]{article}
\usepackage{fullpage}
\usepackage[utf8]{inputenc}
\usepackage[english]{babel}
\usepackage{blindtext}
\usepackage{color}
\usepackage{graphicx}
\usepackage{amsthm}
\usepackage{tikz-cd}
\usepackage{mathtools}
\usepackage{amsfonts}
\usepackage{amssymb,amsmath}
\usepackage{mathtools}
\usepackage{amsfonts}
\usepackage{hyperref}
\usepackage{subcaption}
\usepackage{caption}
\usepackage{amsmath}
\usepackage{tcolorbox}
\usepackage{authblk}
\usepackage{caption}

\newtheorem{theorem}{Theorem}[section]

\newtheorem{lemma}[theorem]{Lemma}
\theoremstyle{definition}

\theoremstyle{definition}

\newtheorem{example}{Example} 

\usepackage{thmtools, thm-restate}

\graphicspath{ {e:/maths/} }
\graphicspath{ {e:/maths/} }
\graphicspath{ {e:/maths/} }
\graphicspath{ {e:/maths/} }
\title{Stability and Convergence Analysis of an Exact Finite Difference Scheme for Fredholm Integro-Differential Equations}  
\author[1]{Mehebub Alam} 
\author[1]{Rajni Kant Pandey}
\affil[1]{Department of Mathematics, Indian Institute of Technology, Kharagpur, 721302, India}
\date{}
\begin{document}
\maketitle

\abstract{This report addresses the boundary value problem for a second-order linear singularly perturbed FIDE. Traditional methods for solving these equations often face stability issues when dealing with small perturbation parameters. We propose an exact finite difference method to solve these equations and provide a detailed stability and $\varepsilon$-uniform convergence analysis. Our approach is validated with an example, demonstrating its uniform convergence and applicability, with a convergence order of 1. The results illustrate the method's robustness in handling perturbation effects efficiently.}
\section*{Keywords}Exact Finite Difference Method, Fredholm integro-differential equations, Singularly perturbed, $\epsilon$-uniform Convergence, fitted difference scheme
\section{Introduction}
Fredholm integro-differential equations (FIDEs) are crucial in many scientific disciplines and practical fields, including physics, mechanics, fluid dynamics, electrostatics, chemistry, biology, astronomy, and more.\cite{kythe2002computational,polyanin2008handbook,cont2005integro}.
In engineering and science, these equations are crucial for modeling various phenomena, leading researchers to develop theoretical frameworks, numerical computations, and analyses for FIDEs.
Several semi-analytical techniques, such as the Legendre polynomial approximation, differential transform method, variational iteration method, and homotopy perturbation method, have been proposed in the literature to solve these FIDEs.  \cite{bildik2010comparison,javidi2007numerical,ziyaee2015differential,lanlege2023solution}. Furthermore, many numerical methods have been proposed in recent times. These include the Galerkin-Chebyshev wavelets method, 
  the Nyström method, extrapolation method, the Sinc-Collocation method, and so on ( see, e.g.,\cite{tahernezhad2020exponential,henka2022numerical,tair2021solving,brezinski2019extrapolation,mohsen2007sinc}). However, these studies have only considered regular cases.
  
  In this report, we consider the boundary value problem of second-order FIDE of the form
  \begin{equation}
  \label{eq:1}
       Lu:=\varepsilon u''(x)+a(x) u'(x)-b(x)u(x)+\lambda \int_0^l K(x,s)u(s)ds=f(x), \ x\in \Omega :=(0,l),
  \end{equation}
subject to the boundary conditions
\begin{equation}
\label{eq:2}
    u(0)=A,\ u(l)=B,
\end{equation}
where, $\varepsilon \in (0, 1]$ is a small parameter, $\Bar{\Omega} = [0, l]$, $\lambda $ is a real parameter, the functions $a(x) \geq \alpha > 0, f(x) \  (x \in \bar{\Omega})$, $u(x)\   (x \in \bar{\Omega}), K(x, s)  \  ( (x, s) \in \bar{\Omega} \times \bar{\Omega})$ are
sufficiently smooth. Under these conditions, the problem (\ref{eq:1}) - (\ref{eq:2}) has unique solution u. Singularly perturbed differential equations are special equations that usually involve a tiny number $\varepsilon$ multiplying the highest order terms in the equation. When we solve these equations, we observe various phenomena happening at different scales. In certain narrow parts of the problem space, some derivatives change much faster than others. These narrow regions with rapid changes are called boundary or interior layers, depending on where they occur. Here $\varepsilon$ tends toward zero, the boundary layers appear in the neighborhood of $x = 0.$ Such equations are common in mathematical problems in the sciences and engineering, for example, the study of moving air and how it affects structures, the behavior of fluids, how electricity behaves in complicated situations, different ways to understand how populations grow, creating models for neural networks, materials that remember their previous state, and mathematical models for how tiny particles move in a chaotic fluid \cite{qin2014integral,ahmad2008some,nieto2007new}. Usual discretization methods for solving problems with very small variations are known to be unstable and often do not provide good solutions when the variations are extremely small. Hence, there is a need to create consistent numerical approaches to tackle such problems.

\section{Literature review} 
In recent years, there has been considerable research on singularly perturbed integro-differential equations, leading to numerous publications that propose various numerical schemes for solving these equations. This literature review highlights several key contributions in the field.

In 2018, Amiraliyev et al. \cite{uniform_amiraliyev_2018} introduced an exponentially fitted difference method on a uniform mesh for first-order equations, demonstrating first-order uniform convergence in $\varepsilon$. In 2020, Iragi and Munyakazi \cite{uniformly_iragi_2020} applied the fitted operator finite difference method, using the right-side rectangle rule and trapezoidal integration on a Shishkin mesh, to Volterra integro-differential equations with singularity, aiming to develop accurate and efficient numerical schemes.

In the same year, Mbroh et al. \cite{second_mbroh_2020} proposed a non-standard finite difference scheme that used the composite Simpson's rule and enhanced convergence order through Richardson extrapolation, focusing on improving accuracy and stability for these equations. Amiraliyev et al. \cite{fitted_amiraliyev_2020} also analyzed a nearly second-order accurate finite difference scheme for first-order SPFIDE.

In 2021, Cimen and Cakir \cite{uniform_cimen_2021} suggested an exponentially fitted difference scheme for singularly perturbed Fredholm integro-differential equations, aiming to capture the behavior of solutions with exponential terms accurately. Durmaz and Amiraliyev \cite{robust_durmaz_2021} developed a second-order homogeneous difference scheme on a Shishkin mesh for Fredholm integro-differential equations with layer behavior to provide accurate numerical solutions.

Amiraliyev et al. \cite{amiraliyev2021numerical} proposed a first-order uniformly convergent fitted finite difference scheme on a uniform mesh for Fredholm integro-differential equations. In 2022, Cakir et al. \cite{cakir2022numerical} derived a numerical approach for an initial-value problem for singularly perturbed nonlinear Fredholm integro-differential equations, achieving uniform convergence in the discrete maximum norm with respect to the perturbation parameter.

Amirali et al. \cite{amirali2023first} in 2022 proposed a first-order numerical method for first-order singularly perturbed nonlinear Fredholm integro-differential equations with integral boundary conditions. In 2023, Durmaz \cite{durmaz2023efficient} applied a fitted finite difference approach using a composite trapezoidal rule to both the integral part of the equation and the initial condition, achieving uniform second-order convergence.

Lastly, in 2023, Amiri \cite{amiri2023effective} developed an effective numerical method for solving a class of nonlinear singular two-point boundary value Fredholm integro-differential equations. By using an appropriate interpolation and a q-order quadrature rule of integration, the problem was approximated by nonlinear finite difference equations, achieving a convergence order of $O(h^{\min\{7/2,q-1/2\}})$ in the $L_2$ norm, where $q$ is the order of the quadrature method. 
 The goal of this report is to design a homogeneous type difference scheme for the problem (\ref{eq:1})-(\ref{eq:2}). Our approach employs the exact (non-standard) finite-difference method to approximate the differential part together with the composite trapezoidal rule for the integral part to discretize the SPFIDE on a uniform mesh. 

 This study is organized subsequently. In Section 2, we report some preliminary work that is relevant to the study. Section 3, proposes a difference scheme for SPFIDE. Later, in Section 4 we establish the error analysis for the scheme. Then, in Section 5, we present some numerical results that illustrate the scheme's performance. Finally, we conclude our work in Section 6.
 In this report, we will use the symbol $C$ to represent a generic constant that is independent of $\varepsilon$ and the mesh parameter. The notation $\|g\|_{\infty}$ signifies the maximum norm for any continuous function $g(x)$ over the associated closed interval. Also Let us define the forward, backward, and second-order central finite difference operators as
$$D^{+} u_{j}=\frac{u_{j+1}-u_{j}}{h}, D^{-} u_{j}=\frac{u_{j}-u_{j-1}}{h}, D^{+} D^{-} u_{j}=\frac{D^{+} u_{j}-D^{-} u_{j}}{h}$$.
\section{Numerical discretization}

The theoretical basis of the nonstandard discrete numerical method is based on the development of the exact finite difference method. The author of \cite{mickens2005advances} presented techniques and rules for developing nonstandard finite difference methods for different problem types. In Mickens's rules, to develop a discrete scheme, the denominator function for the discrete derivatives must be expressed in terms of more complicated functions of step sizes than those used in the standard procedures. These complicated functions constitute a general property of the schemes, which is useful while designing reliable schemes for such problems.
For the problem of the form in (\ref{eq:1})-(\ref{eq:2}), to construct the exact finite difference scheme. 
Consider the constant coefficient homogeneous sub-problems corresponding to (\ref{eq:1}) are
\begin{gather}
\varepsilon u^{\prime \prime}(x)+a u^{\prime}(x)-b u(x)=0  \label{eq4}\\
\varepsilon u^{\prime \prime}(x)+a u^{\prime}(x)=0 \label{eq5}
\end{gather}
where $a(x) \geq a$ and $b(x) \geq b$. Two linear independent solutions of (\ref{eq4}) are $\exp \left(\lambda_{1} x\right)$ and $\exp \left(\lambda_{2} x\right)$, where
\begin{equation}
\lambda_{1,2}=\frac{-a \pm \sqrt{a^{2}+4 \varepsilon b}}{2 \varepsilon} \label{eq6}
\end{equation}
We discretize the domain $[0,l]$ using the uniform mesh length $\Delta x=h$ such that $\Omega=\left\{x_{i}=x_{0}+i h, 1,2, \ldots, N, x_{0}=0, x_{N}=l, h=\frac{l}{N}\right\}$, where $N$ denotes the number of mesh points. We denote the approximate solution to $u(x)$ at the grid point $x_{i}$ by $u_{i}$. Now our main objective is to calculate the difference equation, which has the same general solution as the differential equation (\ref{eq4}) has at the grid point $x_{i}$ given by $u_{i}=A \exp \left(\lambda_{1} x_{i}\right)+B \exp \left(\lambda_{2} x_{i}\right)$. Using the theory of difference equations, we have

\begin{equation}
\label{eq7}
det\left[\begin{array}{ccc}
u_{i-1} & \exp \left(\lambda_{1} x_{i-1}\right) & \exp \left(\lambda_{2} x_{i-1}\right)  \\
u_{i} & \exp \left(\lambda_{1} x_{i}\right) & \exp \left(\lambda_{2} x_{i}\right) \\
u_{i+1} & \exp \left(\lambda_{1} x_{i+1}\right) & \exp \left(\lambda_{2} x_{i+1}\right)
\end{array}\right]=0
\end{equation}
Simplifying (\ref{eq7}), we obtain
\begin{equation}
-\exp \left(-\frac{a h}{2 \varepsilon}\right) u_{i-1}+2 \cosh \left(\frac{h \sqrt{a^{2}+4 \varepsilon b}}{2 \varepsilon}\right) u_{i}-\exp \left(\frac{a h}{2 \varepsilon}\right) u_{i+1}=0 \label{eq8}
\end{equation}
which is an exact difference scheme for (\ref{eq4}).
After doing the arithmetic manipulation and rearrangement on (\ref{eq8}), for the constant coefficient problem (\ref{eq5}), we get
\begin{equation}
\varepsilon \frac{u_{i-1}-2 u_{i}+u_{i+1}}{\frac{h \varepsilon}{a}\left(\exp \left(\frac{a h}{\varepsilon}\right)-1\right)}+a \frac{u_{i+1}-u_{i}}{h}=0 \label{eq9}
\end{equation}
The denominator function becomes $\psi^{2}=\frac{h \varepsilon}{a}\left(\exp \left(\frac{h a}{\varepsilon}\right)-1\right)$. Adopting this denominator function for the variable coefficient problem, we write it as
\begin{equation}
\psi_{i}^{2}=\frac{h \varepsilon}{a_{i}}\left(\exp \left(\frac{h a_{i}}{\varepsilon}\right)-1\right) \label{eq10}
\end{equation}
where $\psi_{i}^{2}$ is the function of $\varepsilon, a_{i}$, and $h$.
By using the denominator function $\psi_{i}^{2}$ into the main scheme, we obtain the difference scheme as
\begin{equation}
L_{\varepsilon}^{N} u_{i} \equiv \varepsilon \frac{u_{i+1}-2 u_{i}+u_{i-1}}{\psi_{i}^{2}}+a_{i} \frac{u_{i+1}-u_{i}}{h}-b_{i} u_{i}+\lambda \int_{0}^{l} K(x_{i},s) u(s) d s=f_{i} -R_1 \label{eq11},
\end{equation}
where $R_1$ is given by
$$
\begin{aligned}
& R_1=f_{i}-L_{\varepsilon}^{N} u_{i} \\
& \quad=\left(\varepsilon u_{i}^{\prime \prime}+a_{i} u'_{i}-b_i u_i+\lambda \int_{0}^{l} K(x_{i},s) u(s) d s\right) \\
& \quad-\left(\varepsilon \frac{u_{i-1}-2 u_{i}+u_{i+1}}{\psi_{i}^{2}}+a_{i}\frac{u_{i+1}-u_{i}}{h}-b_iu_i +\lambda \int_{0}^{l} K_{i}(s) u(s) d s\right) \\
& \quad=\varepsilon \left(u_{i}^{\prime \prime}- \frac{u_{i-1}-2 u_{i}+u_{i+1}}{\psi_{i}^{2}}\right) +a_i\left( u'_i-\frac{u_{i+1}-u_{i}}{h}  \right)\\
&=\varepsilon\left(\frac{d^{2}}{d x^{2}}-\frac{h^2 D^{+} D^{-}}{\psi_{i}^{2}}\right) u_{i}+a_{i}\left(\frac{d}{d x}-D^{+}\right) u_{i}
\end{aligned}
$$
Moreover applying the composite trapezoidal rule to the integral term in (\ref{eq11}). To drive composite trapezoidal rule substitute $n=1$ in Newton-Cotes quadrature formulae in \cite{mickens2005advances} obtain the following result
$$
\begin{aligned}
& \int_{s_{0}}^{s_{1}} K\left(x_{i}, s\right) u(s) d s \\
& =\frac{h}{2}\left[K\left(x_{i}, s_{0}\right) u\left(s_{0}\right)+K\left(x_{i}, s_{1}\right)u\left(s_{1}\right)\right]
\end{aligned}
$$
where $K\left(x_{i}, s_{j}\right)=K_{i j} \ u\left(s_{j}\right)=u_{j}$ for $j=0,1, \ldots, n$. Now, the error of this method can be approximated in the following way.
Let the function $y=K\left(x_{i}, s\right) u(s)$ be continuous and possess a continuous derivatives in $\left[s_{0}, s_{1}\right]$. Expanding $y$ about $s=s_{0}$ we obtain\\

\begin{align}\nonumber
& y(x)=y_{0}+\left(s-s_{0}\right) y_{0}^{\prime}+\frac{1}{2}\left(s-s_{0}\right)^{2} y_{0}^{\prime \prime}+\frac{1}{3!}\left(s-s_{0}\right)^{3} y_{0}^{\prime \prime \prime}+\ldots \\\nonumber
& \int_{s_{0}}^{s_{1}} K\left(x_{i}, s\right) u(s) d s=\int_{0}^{1} h\left(y_{0}+p h y_{0}^{\prime}+\frac{(p h)^{2}}{2} y_{0}^{\prime \prime}+\frac{(p h)^{3}}{3!} y_{0}^{\prime \prime \prime}+\ldots\right) d p \\\nonumber
& \quad=h\left[p y_{0}+\frac{p^{2} h}{2} y_{0}^{\prime}+\frac{(p h)^{2}}{6} p y_{0}^{\prime \prime}+\frac{(p h)^{3}}{24} p y_{0}^{\prime \prime \prime}+\ldots\right]_{0}^{1} \\
& \quad=h y_{0}+\frac{h^2}{2}y_{0}^{\prime}+\frac{h^{3}}{6} y_{0}^{\prime \prime}+\frac{h^{4}}{24} y_{0}^{\prime \prime \prime}+\ldots \label{eq12}
\end{align}
Therefore
\begin{align}
& y_{0}=y_{0} \label{eq13} \\
& y_{1}=y_{0}+h y_{0}^{\prime}+\frac{h^{2}}{2} y_{0}^{\prime \prime}+\frac{h^{3}}{6} y_{0}^{\prime \prime \prime}+\ldots, \label{eq14} 
\end{align}
From (\ref{eq13}), (\ref{eq14}), we get:
\begin{align}
\frac{h}{2}\left[y_{0}+y_{1}\right]
& =h y_{0}+\frac{h^2}{2}y_{0}^{\prime}+\frac{h^{3}}{4} y_{0}^{\prime \prime}+\frac{h^{4}}{12} y_{0}^{\prime \prime \prime}+\ldots \label{eq15} 
\end{align}
From (\ref{eq12}) and (\ref{eq15}) we obtain,
$$
\int_{s_{0}}^{s_{1}} y d s-\frac{h}{2}\left[y_{0}+y_{1}\right]=\frac{-1}{12} h^{3} y_{0}^{'''}
$$
This is the error committed in the interval $\left[s_{0}, s_{1}\right]$.
Let $[0, l]$ be sub-divided into $N$ number of sub-divisions, $0=s_{0}<s_{1}<s_{2}<\cdots<s_{N}=l$, the integral over the whole interval is found by adding these integrations and is equal to
$$\int_{s_0}^{s_N}y ds=\frac{h}{2}\left( \sum_{j=1}^{N}\left( K_{ij-1}y_{j-1}+K_{ij}y_{j}\right)\right)$$
We obtain the errors in the interval $[0,l]$ as
\begin{equation}
    R_2=-\frac{1}{12}h^3\left[ y^{'''}_0+y^{'''}_1+\cdots+y^{'''}_{N-1}\right]=-\frac{lh^2}{12}u^{''}(\xi) \label{eq16}
\end{equation}
where $u^{''}(\xi)$ is the largest value of the $N$-quantities on 2nd derivatives. Therefore, the integral term in (\ref{eq11}) is approximated as:
\begin{align} \label{eq17}
\int_{0}^{l} K\left(x_{i}, s\right) u(s) d s= &\frac{h}{2}\left( \sum_{j=1}^{N}\left( K_{ij-1}u_{j-1}+K_{ij}u_{j}\right)\right) +R_{2} 
\end{align}
From (\ref{eq11}) and (\ref{eq17}) for $i=1,2, \ldots N-1$, we have the following relation
\begin{equation}  
L_{\varepsilon}^{N} u_{i}:=\varepsilon \frac{u_{i-1}-2 u_{i}+u_{i+1}}{\psi_{i}^{2}}+a_{i}\frac{u_{i+1}-u_{i}}{h} -b_iu_i+\lambda \frac{h}{2}\left( \sum_{j=1}^{N}\left( K_{ij-1}u_{j-1}+K_{ij}u_{j}\right)\right)=f_{i}-R, \label{eq18}
\end{equation}
where $R=-R_{1}-R_{2}$\\
Based on (\ref{eq18}) we propose the following difference scheme for approximating (\ref{eq:1}).
\begin{align} \nonumber
&L_{\varepsilon}^{N} u_{i}:=\varepsilon \frac{u_{i-1}-2 u_{i}+u_{i+1}}{\psi_{i}^{2}}+a_{i}\frac{u_{i+1}-u_{i}}{h} -b_iu_i+\lambda \frac{h}{2}\left( \sum_{j=1}^{N}\left( K_{ij-1}u_{j-1}+K_{ij}u_{j}\right)\right)=f_{i}\\
&u_0=A,\ u_N=B. \label{eq19}
\end{align}
Lastly, from (\ref{eq19}) the linear system equations for $u_{1}, u_{2}, u_{3}, \ldots, u_{N-1}$ are generated. Therefore, the generated system of linear algebraic equations can be written in a matrix form of\\
\begin{align}
&(M+S) u=F \label{eq20}
\end{align}
where $M$ and $S$ are coefficient matrix, $F$ is a given function and $u$ is an unknown function which is to be determined. The entries of $M, S$ and $F$ are given as:
$$
\begin{aligned}
& M= \begin{cases}a_{i i}=-\frac{2 \varepsilon}{\psi_{i}^{2}}+a\left(x_{i}\right)-b\left(x_{i}\right), & \text { for } i=1,2, \ldots, N-1 \\
a_{i i+1}=\frac{\varepsilon}{\psi_{i}^{2}}+\frac{a\left(x_{i}\right)}{h}, & \text { for } i=1,2, \ldots, N-2, \\
a_{i i-1}=\frac{\varepsilon}{\psi_{i}^{2}}, & \text { for } i=2,3, \ldots, N-1,\end{cases} \\
& S=\left\{\lambda h K_{ij},\ i,j=1,2,\cdots,N-1\right.
\end{aligned}
$$
and

$F=\left\{\begin{array}{l}f_{1}-\left(\left(\frac{1}{2}\lambda h K_{1,0}+\frac{\varepsilon}{\psi_{1}^{2}}\right) A+\frac{1}{2}\lambda h K_{1, N} B\right), \\ f_{i}-\frac{1}{2}\lambda h\left( K_{i, 0} A+ K_{i, N} B\right), \text { for } i=2,3, \ldots, N-2, \\ f_{N-1}-\left(\frac{1}{2}\lambda h K_{N-1,0} A+\left(\frac{1}{2}\lambda h  K_{N-1, N}+\frac{\varepsilon}{\psi_{N}^{2}}\right) B\right) .\end{array}\right.$

\section{Stability and convergence analysis}
In this section, we need to show the discrete scheme in (\ref{eq19}) satisfies the discrete maximum principle, uniform stability estimates, and uniform convergence. The difference operator, $L_{\varepsilon}^{N}$ satisfies the following lemma.

\begin{lemma}
(Discrete Minimum Principle). 
Let $u_{i}$ be any mesh function that satisfies $u_{0} \geq 0, u_{N} \geq 0$, and $L_{\varepsilon}^{N} u_{i} \leq 0, i=1,2, \ldots, N-1$. Then $u_{i} \geq$ $0, i=1,2, \ldots, N$.  
\end{lemma}
\begin{proof}
The proof is obtained by contradiction. Let $j$ be such that $u_{j}=$ $\min u_{i}$, and suppose that $u_{j}<0$. Clearly, $j \notin\{0, N\}, u_{j+1}-u_{j} \geq 0$, and $u_{j}-u_{j-1} \leq 0$. Therefore,
$$
\begin{aligned}
L_{\varepsilon}^{N} u_{j} & =\frac{\varepsilon}{\psi_{j}^{2}}\left(u_{j+1}-2 u_{j}+u_{j-1}\right)+\frac{a_{j}}{h}\left(u_{j+1}-u_{j}\right)-b_j u_{j} +\lambda \frac{h}{2}\left( \sum_{p=1}^{N}\left( K_{jp-1}u_{p-1}+K_{jp}u_{p}\right)\right)\\
& =\frac{\varepsilon}{\psi_{j}^{2}}\left[\left(u_{j+1}-u_{j}\right)-\left(u_{j}-u_{j-1}\right)\right]+\frac{a_{j}}{h}\left(u_{j+1}-u_{j}\right)-b_j u_{j}+\lambda \frac{h}{2}\left( \sum_{p=1}^{N}\left( K_{jp-1}u_{p-1}+K_{jp}u_{p}\right)\right) \geq 0
\end{aligned}
$$
where the strict inequality holds if $u_{j+1}-u_{j}>0$. This is a contradiction and therefore $u_{j} \geq 0$. Since $j$ is arbitrary, we have $u_{i} \geq 0, \quad i=1,2, \ldots, N$.
\end{proof}
The discrete minimum principle enables us to prove the next lemma which provides the boundedness of the solution.
\begin{lemma} \label{l1}
If $u_{i}$ is the solution of the discrete problem (\ref{eq19}), then it admits the bound
$$
\left|u_{i}\right| \leq \beta^{-1} \max _{x_{i} \in[0, l]}\left|L_{\varepsilon}^{N} u_{i}\right|+\max \{|A|,|B|\}
$$
\end{lemma}
\begin{proof}
We consider the functions $\psi^{ \pm}$defined by $\psi_{i}^{ \pm}=p \pm u_{i}$, where $p=\beta^{-1} \max _{x_{i} \in[0, l]}\left|L_{\varepsilon}^{N} u_{i}\right|+\max \{|A|,|B|\}$. At the boundaries we have
$\psi_{0}^{ \pm}=p \pm u_{0}=p \pm A \geq 0, \psi_{N}^{ \pm}=p \pm u_{N}=p \pm B \geq 0$.
Now for $\Omega_{N}$ we have
$$
\begin{aligned}
L_{\varepsilon}^{N} \psi_{i}^{ \pm}= & \varepsilon\left(\frac{p \pm u_{i+1}-2\left(p \pm u_{i}\right)+p \pm u_{i-1}}{\psi_{i}^{2}}\right) \\
& +a_{i}\left(\frac{(p \pm u_{i+1})-(p \pm u_{i})}{h}\right)-b_i \left( p \pm u_i\right) +\left( p \pm \lambda \frac{h}{2}\left( \sum_{j=1}^{N}\left( K_{ij-1}u_{j-1}+K_{ij}u_{j}\right)\right)\right)\\
= & -b_{i} p \pm L_{\varepsilon}^{N} u_{i} \\
= & -b_{i}\left(\beta^{-1} \max _{x_{i} \in[0, l]}\left|L_{\varepsilon}^{N} u_{i}\right|+\max (|A|,|B|)\right) \\
& \pm L_{\varepsilon}^{N} u_i \leq 0, \text { since } b_{i} \geq \beta >0
\end{aligned}
$$
From Lemma \ref{l1} it follows that $\psi_{i}^{ \pm} \geq 0, \forall x_{i} \in[0, l]$, which completes the proof.
\end{proof}
\begin{lemma} \cite{mbroh2022robust}\label{lemma3}
For a fixed mesh and for $\varepsilon \rightarrow 0$, it holds
$$
\begin{gathered}
\lim _{\varepsilon \rightarrow 0} \max _{1 \leq i \leq N-1}\left(\frac{\exp \left(\frac{-a x_{i}}{\varepsilon}\right)}{\varepsilon^{m}}\right)=0, \quad m=1,2,3, \ldots \\
\end{gathered}
$$
where $x_{i}=i h, h=\frac{1}{N}, i=1,2, \ldots, N-1$.   
\end{lemma}
\begin{proof}
Consider the partition $[0,l]:=\left\{0=x_{0}<x_{1}<\cdots<x_{N-1}<x_{N}=\right.$ $l\}$. For the interior grid points, we have
$$
\begin{aligned}
& \max _{1 \leq i \leq N-1} \frac{\exp \left(\frac{-a x_{i}}{\varepsilon}\right)}{\varepsilon^{m}} \leq \frac{\exp \left(\frac{-a x_{1}}{\varepsilon}\right)}{\varepsilon^{m}}=\frac{\exp \left(\frac{-a h}{\varepsilon}\right)}{\varepsilon^{m}} \\
& \text { as } x_{1}=l-x_{N-1}=h
\end{aligned}
$$
Then, applying L'Hospital's rule $m$ times gives
$$
\lim _{\varepsilon \longrightarrow 0} \frac{\exp \left(\frac{-a h}{\varepsilon}\right)}{\varepsilon^{m}}=\lim _{r=\frac{1}{\varepsilon} \longrightarrow \infty} \frac{r^{m}}{\exp (a h r)}=\lim _{r=\frac{1}{\varepsilon} \longrightarrow \infty} \frac{m!}{(a h)^{m} \exp (a h r)}=0
$$
\end{proof}
Next, we analyze the uniform convergence of the method. From (\ref{eq18}) and (\ref{eq19}) for the error of the approximate solution, $z_{i}=u_{i}-u\left(x_{i}\right)$ we have
\begin{align} \nonumber
& L_{\varepsilon}^{N} z_{i}:=\varepsilon \frac{z_{i-1}-2 z_{i}+z_{i+1}}{\psi_{i}^{2}}+a_{i}\frac{z_{i+1}-z_{i}}{h}-b_iz_i+\lambda \frac{h}{2}\left( \sum_{j=1}^{N}\left( K_{ij-1}z_{j-1}+K_{ij}z_{j}\right)\right), \\
&    z_{0}=0, z_{N}=0, i=1,2, \ldots, N-1, \label{eq21}
\end{align}
\begin{theorem}
If $a, b, f \in C^{2}[0, l], \frac{\partial^{s} K}{\partial x^{s}} \in C[0, l]^{2}$, $(s=0,1,2)$  and $|\lambda|<\frac{\alpha}{\max _{1 \leq i \leq N} \sum_{j=0}^{N} h \eta_{j}\left|K_{i j}\right|}$, where $\eta_j= \begin{cases}1/2, j=0,N \\
1,\ \text{if}\ j=1,2,\ldots,N-1\end{cases} $ the solution U of $(\ref{eq19})$ converges $\varepsilon$-uniformly to the solution $u$ of (\ref{eq:1}). For the error of approximate solution, the following bound holds:
$$
\|U-u\|_{\infty} \leq C h
$$   
\end{theorem}
\begin{proof}
Applying the maximum principle, from (\ref{eq21}) we have
$$
\begin{aligned}
& \|z\|_{\infty} \leq \alpha^{-1}\left\|R-\lambda h \sum_{j=0}^{N} \eta_{j} K_{i j} z_{j}\right\|_{\infty} \\
& \leq \alpha^{-1}\|R\|_{\infty}+|\lambda| \alpha^{-1} \max _{1 \leq i \leq N} \sum_{j=0}^{N} h \eta_{j}\left|K_{i j}\right|\|z\|_{\infty}
\end{aligned}
$$
hence
$$
\|z\|_{\infty} \leq \frac{\alpha^{-1}\|R\|_{\infty}}{1-|\lambda| \alpha^{-1} \max _{1 \leq i \leq N} \sum_{j=0}^{N} h \eta_{j}\left|K_{i j}\right|}
$$
which implies
\begin{equation} \label{eq22}
\|z\|_{\infty} \leq C\|R\|_{\infty} 
\end{equation}
Further, we estimate for $\|R\|_{\infty}$.
The truncation error is given by $R=R_{1}+R_2$ where $R_{1}$ is the error is the discretization of the differential part and $R_2$ is the error in the discretization of the integral part.



\begin{equation*}
R=\varepsilon\left(\frac{d^{2}}{d x^{2}}-\sigma_{r} D^{+} D^{-}\right) u_{i}+a_{i}\left(\frac{d}{d x}-D^{+}\right) u_{i}-\frac{lh^2}{12}u^{''}(\xi) 
\end{equation*}
Rearranging and simplifying it, we obtain
\begin{align} \nonumber
R & =-\varepsilon\left(\frac{h^2}{\psi^2}-1\right) D^{+} D^{-} u_{i}+\varepsilon\left(\frac{d^{2}}{d x^{2}}-D^{+} D^{-} u_{i}\right)+a_{i}\left(\frac{d}{d x}-D^{+}\right) u_{i}-\frac{lh^2}{12}u^{''}(\xi)  \\
& \leq\left|\varepsilon\left(\frac{h^2}{\psi^2}-1\right) D^{+} D^{-} u_{i}\right|+\left|\varepsilon\left(\frac{d^{2}}{d x^{2}}-D^{+} D^{-} u_{i}\right)\right|+\left|a_{i}\left(\frac{d}{d x}-D^{+}\right) u_{i}\right|+\left|\frac{lh^2}{12}u^{''}(\xi) \right| \label{eq23}
\end{align}
We use the estimate $\varepsilon\left|\frac{h^{2}}{\psi^{2}}-1\right| \leq C h$, which can be derived from (\ref{eq10}). Indeed, define $\rho=\frac{a_{i} h}{\varepsilon}, \rho \in(0, \infty)$.Then,
\begin{equation} \label{eq24}
  \varepsilon\left|\frac{h^{2}}{\psi^{2}}-1\right|=a_{i} h\left|\frac{1}{\exp (\rho)-1}-\frac{1}{\rho}\right|=: a_{i} h Q(\rho)  
\end{equation}
By simplifying and writing the above equation explicitly, we obtain
$$
Q(\rho)=\frac{\exp (\rho)-\rho-1}{\rho(\exp (\rho)-1)}
$$
and we obtain that the limit is bounded as

$$
\lim _{\rho \longrightarrow 0} Q(\rho)=\frac{1}{2}, \quad \lim _{\rho \longrightarrow \infty} Q(\rho)=0
$$
Hence, for all $\rho \in(0, \infty)$, we have $Q(\rho) \leqslant C$.
Using the Taylor series expansion, we have the bound

\begin{equation} \label{eq25}
\left|D^{+} D^{-} u_{i}\right| \leq C\left\|u^{\prime \prime}(\xi)\right\|_{\infty},\left|\left(\frac{d}{d x}-D^{+}\right) u_{i}\right| \leq C h^{2}\left\|u^{(3)}(\xi)\right\|_{\infty},\left|\left(\frac{d^{2}}{d x^{2}}-D^{+} D^{-}\right) u_{i}\right| \leq C h^{2}\left\|u^{(4)}(\xi)\right\|_{\infty} 
\end{equation}
for $\xi \in[0,l]$. Using (\ref{eq24}) and (\ref{eq25}) into (\ref{eq23}) the truncation error becomes
\begin{equation} \label{eq26}
R \leq C h\left\|u^{\prime \prime}(\xi)\right\|_{\infty}+\varepsilon C h^2\left\|u^{(4)}(\xi)\right\|_{\infty}+C h^{2}\left\|u^{(3)}(\xi)\right\|_{\infty}
\end{equation}
Using the bound for the derivative of the solution, we obtain
\begin{align} \nonumber
\left\|R\right\|_{\infty} \leq &C h\left(1+\varepsilon^{-2} \max _{x_{i}} \exp \left(\frac{-\alpha x_{i}}{\varepsilon}\right)\right)+\varepsilon C h^{2}\left(1+\varepsilon^{-4} \max _{x_{i}} \exp \left(\frac{-\alpha x_{i}}{\varepsilon}\right)\right)\\ \nonumber
&+C h^{2}\left(1+\varepsilon^{-3} \max _{x_{i}} \exp \left(\frac{-\alpha x_{i}}{\varepsilon}\right)\right) \\\nonumber
= &C h\left(1+\varepsilon^{-2} \max _{x_{i}} \exp \left(\frac{-\alpha x_{i}}{\varepsilon}\right)\right)+ C h^{2}\left(\varepsilon+\varepsilon^{-3} \max _{x_{i}} \exp \left(\frac{-\alpha x_{i}}{\varepsilon}\right)\right)\\
&+C h^{2}\left(1+\varepsilon^{-3} \max _{x_{i}} \exp \left(\frac{-\alpha x_{i}}{\varepsilon}\right)\right) \label{eq27}
\end{align}
Since $\varepsilon^{-3} \geq \varepsilon^{-2}$ we obtain the bound
\begin{equation} \label{eq28}
\left\|R\right\|_{\infty} \leq Ch\left(1+\varepsilon^{-3} \max _{x_{i}} \exp \left(\frac{-\alpha x_{i}}{\varepsilon}\right)\right) 
\end{equation}
Using the result in Lemma \ref{lemma3}, which implies that
\begin{equation} \label{eq29}
\left\|R\right\|_{\infty} \leq C h 
\end{equation}
The bound (\ref{eq29}) together with (\ref{eq22}) completes the proof.
\end{proof}

\section*{Numerical results and discussion}
To verify the established theoretical results in this paper, we perform an experiment using the proposed numerical scheme on the problem of the form given in (\ref{eq:1}). The convergence rate and maximum pointwise errors, which have been calculated, are displayed below in tabular form.
The maximum pointwise error is specified by:
$$E_\epsilon^\mathcal{N}=\|\mathit{v}-y\|_{\infty,\Bar{\Gamma}}$$
where $\mathit{v}$ is the exact solution and $y$ is approximate solution.  In addition, the estimates of the $\epsilon-$uniform maximum pointwise error are derived from:
$$E^\mathcal{N}=\underset{\epsilon}{\max}\  
 E_\epsilon^\mathcal{N}.$$
Convergence rates are calculated by:
$$P^\mathcal{N}_\epsilon=\frac{\ln{(E^\mathcal{N}_\epsilon/E^{2\mathcal{N}}_\epsilon})}{\ln{2}}.$$
and  $\epsilon-$uniform convergence rates are derived by:
$$P^\mathcal{N}=\frac{\ln{(E^\mathcal{N}/E^{2\mathcal{N}}})}{\ln{2}}.$$
\begin{example} \label{Example1}
Consider the singularly perturbed problem
$$
\begin{aligned}
\varepsilon u^{\prime \prime}+u'+ \int_{0}^{1} x u(s) d s & =-e^{-\frac{x}{\varepsilon}}+ \varepsilon x\left(1-e^{-\frac{x}{\varepsilon}} \right)\\
u(0) & =1, u(1)=e^{-\frac{1}{\varepsilon}}
\end{aligned}
$$
\end{example}
\begin{table}[h!] \label{table1}
    \centering
    \caption{Results for 
 Example \ref{Example1}: Convergence rates and maximum pointwise errors on $\Bar{\Gamma}_\mathcal{N}$.} 
    \label{tab:table1}
\begin{tabular}{llllll}
\hline
$\epsilon$  & $\mathcal{N}=64$ & $\mathcal{N}=128$ & $\mathcal{N}=256$ & $\mathcal{N}=512$ & $\mathcal{N}=1024$ \\
\hline
$2^{0}$  & $7.803e-07$ & $1.951e-07$ & $4.877e-08$ & $1.220e-08$ & $ 3.042e-09$ \\
 & $2.00$ & $2.00$ & $2.00$ & $2.00$ &  \\
$2^{-6}$  & $4.869e-04$ & $1.236e-04$ & $3.102e-05$ & $7.764e-06$ & $1.942e-06$ \\
 & $1.98$ & $1.99$ & $2.00$ & $2.00$ &  \\
$2^{-12}$  & $2.962e-03$ & $1.452e-03$ & $6.818e-04$ & $2.934e-04$ & $1.052e-04$ \\
 & $1.03$ & $1.09$ & $1.22$ & $1.48$ &  \\
$2^{-18}$  & $3.056e-03$ & $1.547e-03$ & $7.776e-04$ & $3.893e-04$ & $1.942e-04$ \\
 & $0.98$ & $0.99$ & $1.00$ & $1.00$ &  \\ 
$2^{-24}$  & $3.057e-03$ & $1.548e-03$ & $7.792e-04$ & $3.908e-04$ & $1.957e-04$ \\ \vspace{0.4cm}
  & $0.98$ & $0.99$ & $1.00$ & $1.00$ &  \\
  
$e^\mathcal{N}$ & $3.057e-03$ & $1.548e-03$ & $7.792e-04$ & $3.908e-04$ & $1.957e-04$ \\
 $p^\mathcal{N}$ & $0.98$ & $0.99$ & $1.00$ & $1.00$ &  \\
\hline
\end{tabular}
\end{table}

\section{Conclusion}
In this study, we implemented the exact finite difference method to solve second-order linear singularly perturbed Fredholm Integro-differential equation. The stability and $\varepsilon$-uniform convergence analysis of the proposed scheme are proven. One example is used to investigate the applicability of the scheme. The effect of the perturbation parameter on the solution of the problem is shown in the table. It is demonstrated that the proposed method is uniformly convergent, with an order of convergence of 1.



\bibliographystyle{elsarticle-num} 
\bibliography{bibliography_Enhance}
\end{document}